\documentclass[12pt]{amsart}

\usepackage{latexsym}
\usepackage{amsmath}
\usepackage{amsfonts}
\usepackage{amssymb}

\newenvironment{claim}[1]{\par\emph{Claim:}\space#1}{}
\newenvironment{claimproof}[1]{\par\emph{Proof of Claim:}\space#1}

\newtheorem{theorem}{Theorem}[section]
\newtheorem{lemma}[theorem]{Lemma}
\newtheorem{corollary}[theorem]{Corollary}

\theoremstyle{definition}

\newtheorem{remark}[theorem]{Remark}

\numberwithin{equation}{section}

\begin{document}

\title[Lov\'asz Local Lemma]{An introduction to the probabilistic method through the Lov\'asz Local Lemma}

\author{Irfan Alam}
\address{Department of Mathematics\\
Louisiana State University\\
Baton Rouge, Louisiana}
\email{ialam1@lsu.edu}

\date{April 30, 2015}
\begin{abstract}
We illustrate the use of probability theory in existential proofs, focusing on the Lov\'asz Local Lemma. This result gives a lower bound for the probability of avoiding a suitable finite collection of events. We describe some applications of this result in hypergraph packing and Latin transversals.

\end{abstract}

\maketitle
\vspace{-4mm}
\section{Introduction}
\label{introduction}
The probabilistic method uses probability theory to prove existence of a non-random object. Typically, an appropriate probability measure is defined on a suitable space consisting of all objects under consideration. If the probability of an event is shown to be non-zero, then that event must be non-empty (as a set). In particular, if such an event is defined in terms of the objects satisfying certain properties, then this proves the existence of such objects. Though the probabilistic method was initially used primarily in combinatorics, it has since been refined and used in a number of other areas, such as number theory, linear algebra, real analysis and computer science. So many different ways of using probability in these areas have been found that it is now impossible to narrow the probabilistic method down to a few techniques of proof. In this paper, we will focus on one specific technique that exploits restrictions on interdependence of certain events to guarantee the possibility of avoiding all those events. The main result used in this approach is called the Lov\'asz Local Lemma. The simplest form of the result was first introduced by Erd\H{o}s and Lov\'asz to prove results on hypergraph coloring \cite{Erdos/Lovasz}. The interested reader may refer to \cite{Alon/Spencer} for a thorough treatment of other techniques of the probabilistic method. 

We will first prove a general version of the lemma as stated in \cite{Lu/Szekely}. Our proof will be a modification of the proof of a slightly weaker form of the lemma in \cite{Alon/Spencer}. Unlike the latter form, the version of Lovasz Local Lemma that we will consider works even if there is \textbf{no} independence among the events. This version is sometimes called the Lopsided Lovasz Local Lemma. It was first introduced by Erd\H{o}s and Spencer in \cite{Erdos/Spencer} to prove a result on Latin transversals. After giving a proof of the lemma, we will describe its applications in hypergraph packings using the machinery of negative dependency graphs on the space of random injections developed in \cite{Lu/Szekely}. We will finally use this new machinery to give a proof of a slightly stronger form of the result on Latin transversals from \cite{Erdos/Spencer}.

\section{The Lemma}
\label{The Lemma}
We first introduce the basic concepts needed from probability and graph theory. A \textbf{probability space} is a triple containing a set, a sigma algebra on that set, and a measure defined on that sigma algebra that assigns the value 1 to the original set. Throughout this section, the probability space is $(\Omega, \mathcal{F}, P)$. Elements of the sigma algebra $\mathcal{F}$ are called the \textbf{events} of the space. A \textbf{Boolean combination} of finitely many events is an element in the smallest algebra generated by these events. Two events $A$ and $B$ are said to be independent if $P(A \cap B) = P(A)P(B)$. A finite collection of events is called \textbf{mutually independent} if any event in that collection is independent of any Boolean combination of the other events in that collection. If $A$ and $B$ are events such that $P(B) \neq 0$, then the \textbf{conditional probability} of $A$ given $B$ is defined as $P(A|B) := \frac{P(A \cap B)}{P(B)}$. The probability measures that we will consider in later sections will always be defined on the sigma algebras of power sets of finite sets. Unless otherwise specified, these measures will always be the \textbf{uniform probability measure}, that is, they will assign the same non-zero probability to all singleton sets in the space.

A \textbf{graph} $G$ is a pair $(V,E)$, where $V$ is a set and $E$ is a collection of $2$-subsets of $V$, that is, $E$ contains subsets of $V$ that have cardinality $2$. The elements of $V$ are called the \textbf{vertices} of the graph $G$, and the elements of $E$ are called the \textbf{edges} of the graph $G$. A vertex of a graph is said to be \textbf{incident} on an edge, if the edge contains that vertex. For a vertex $v$ of a graph $(V,E)$, let $J_v := \{w \in V: \{v,w\} \in E\}$. The number of edges that contain a vertex is called the \textbf{degree} of that vertex. Hence, the degree of a vertex $v$ is equal to the cardinality of $J_v$. The maximum degree of vertices in a graph is called the \textbf{degree} of the graph. 
 
A graph $G$ on $[n] := \{1,\ldots,n\}$ is called a \textbf{negative dependency graph} for the events $A_1, \ldots , A_n$ if, for any $i \in [n]$ and any $S \subseteq {J_i}^c$ with $P(\cap_{j \in S}{A_j}^c) > 0$, we have
\begin{equation}\label{first cond}
P(A_i | \cap_{j \in S}{A_j}^c) \leq P(A_i).
\end{equation}

Note that this is trivially satisfied if $P(A_i) = 0$, and if $P(A_i) > 0$, then the above condition is equivalent to the condition
\begin{equation}\label{second cond}
P(\cap_{j \in S}{A_j}^c | A_i) \leq P(\cap_{j \in S}{A_j}^c) \text{ for any such } S \subseteq J_i.
\end{equation}

We are now ready to state the Lovasz Local Lemma.
\\
\begin{lemma}[Lov\'asz Local Lemma]~\\
Suppose that $A_1, \ldots, A_n$ are events and that $G$ is a negative dependency graph for these events. Suppose that $x_1, \ldots, x_n \in [0,1)$ are such that
\begin{equation}\label{LLL1}
P(A_i) \leq x_i \cdot \prod_{j \in J_i}(1 - x_j),
\end{equation}
then $P(\cap_{i=1}^{n}{A_i}^c) \geq \prod_{i = 1}^{n} (1-x_i) > 0$.
\end{lemma}
The following immediate corollary will be useful in our applications.
\begin{corollary} Suppose that $A_1, \ldots, A_n$ are events and let $G$ be a negative dependency graph for these events. Suppose that the degree of any vertex in $G$ is at most $d$. If $p \in [0,1)$ is such that $P(A_i) \leq p$ for all $i \in [n]$ and $ep(d+1) \leq 1$, then $P(\cap_{i=1}^{n}{A_i}^c) > 0$.
\end{corollary}
\begin{proof}
Let $x_i = \frac{1}{d+1}$ for each $i \in [n]$. The inequality $ep(d+1) \leq 1$ implies that for any $i$, 
\vspace{-2mm}
\begin{equation*}
P(A_i) \leq p \leq \frac{1}{e(d+1)} \leq x_i \cdot \frac{1}{e} \leq x_i \cdot \left(1 - \frac{1}{d+1}\right)^d \leq x_i \cdot \left(1 - \frac{1}{d+1}\right)^{|J_i|},
\end{equation*}
since the function $x \mapsto \left(1 - \frac{1}{x+1}\right)^x$ is decreasing on $(0, \infty)$ and converges to $\frac{1}{e}$ as $x \rightarrow \infty$. Hence Lemma 2.1 applies.
\end{proof}
Before giving a proof of the lemma, suppose that $G$ is a graph on $[n]$ such that for any $i \in [n]$, the event $A_i$ is mutually independent of all ${A_j}$ such that $\{i,j\} \notin E(G)$. Such a graph is called a \textbf{dependency graph} for the events $A_1, \ldots, A_n$. A dependency graph clearly also satisfies the conditions of being a negative dependency graph. Many applications of Lemma 2.1 and Lemma 2.2 deal with the setting of a dependency graph. 
\begin{proof}[Proof of Lemma 2.1]
For each $k \in [n]$, let $B_k = \cap_{i=1}^{k}{A_i}^c$. We will show by induction on $k$ that
\begin{equation}\label{hypothesis}
P(B_k) \geq \prod_{i = 1}^{k} (1-x_i) > 0. 
\end{equation}
Note that
\begin{align*}
\vspace{-2mm}
P(B_1) = P({A_1}^c) = 1 - P(A_1) \geq 1 - x_1 \prod_{j \in J_i}(1 - x_j) \geq 1 - x_1 > 0 .
\end{align*}
So the result is true for $k = 1$.
Suppose it is true for each $k \in \{1, \ldots, m-1\}$, where $m \in \{2, \ldots, n\}$.
Then, we have $P(B_{m-1}) \neq 0$, so that 
\begin{align*}
P(B_m) &= P({A_m}^c \cap (\cap_{i = 1}^{m-1}{A_i}^c)) \\
&= P(\cap_{i = 1}^{m-1}{A_i}^c) - P({A_m} \cap (\cap_{i = 1}^{m-1}{A_i}^c)) \\
&= P(B_{m-1}) - P(A_m \cap B_{m-1}) \\
&= P(B_{m-1})\cdot\left[1 - \frac{P(A_m \cap B_{m-1})}{P(B_{m-1})}\right]\\
&\geq \left[\prod_{i=1}^{m-1}(1-x_i)\right]\cdot\left[1 - P(A_m | B_{m-1})\right]. 
\end{align*}
Thus to complete the proof, it suffices to show that $1 - P(A_m | B_{m-1}) \geq 1 - x_m$, or equivalently, that $P(A_m | B_{m-1}) \leq x_m$.
To show this, we will prove the following stronger claim:

\begin{claim}
Let $S$ be a proper subset of $[n]$ such that $P\left(\cap_{j \in S}{A_j}^c\right) > 0$. Then, for any $i \notin S$,
\vspace{-3mm} 
\begin{equation}\label{second induct}
\vspace{-7mm}
P\left(A_i | \cap_{j \in S}{A_j}^c\right) \leq x_i.
\end{equation}
\end{claim}
\begin{claimproof}
We prove this by induction on $|S|$. If $|S| = 0$, then $P(A_i|\cap_{j \in S}{A_j}^c) = P(A_i | \Omega) = P(A_i) \leq x_i \cdot \prod_{j \in J_i}(1 - x_j) \leq x_i$. Hence the hypothesis is true when $|S| = 0$. Suppose that it is true whenever $|S| \leq s - 1$, where $s \in \{1, \ldots, n-1\}$. Let $S \subseteq [n]$ be a set with $|S| = s$ and $P(\cap_{j \in S}{A_j}^c) > 0$. Let $i \notin S$ and define
\vspace{-2mm}
\begin{equation*}
\vspace{-2mm}
S_1 = \{j \in S: \{i,j\} \in E(G)\} \text{ and } S_2 = \{j \in S: \{i,j\} \notin E(G)\}.
\end{equation*}
If $S_1 = \emptyset$, then $P\left(A_i | \cap_{j \in S}{A_j}^c\right) \leq P(A_i) \leq x_i$, by definition of negative dependency graphs. So suppose $|S_1| \geq 1, \text{ say } S_1 = \{j_1, \ldots, j_r\}$, where $1 \leq r \leq s$. Note that $P(\cap_{j \in S}{A_j}^c) > 0 \text{ implies } P(\cap_{j \in S'}{A_j}^c) > 0 \text{ for any } S' \subseteq S$. Hence we have
\begin{align*}
P(A_i | \cap_{j \in S}{A_j}^c) &= \frac{\left(\frac{P(A_i \cap (\cap_{j \in S_1}{A_j}^c) \cap (\cap_{j \in S_2}{A_j}^c))}{P(\cap_{j \in S_2}{A_j}^c)}\right)}{\left(\frac{P((\cap_{j \in S_1}{A_j}^c) \cap (\cap_{j \in S_2}{A_j}^c))}{P(\cap_{j \in S_2}{A_j}^c)}\right)} \\
&\leq \frac{P(A_i | \cap_{j \in S_2} {A_j}^c)}{P(\cap_{j \in S_1} {A_j}^c | \cap_{j \in S_2} {A_j}^c)} \\
&\leq \frac{x_i \prod_{j \in S_1}(1 - x_j)}{P(\cap_{j \in S_1} {A_j}^c | \cap_{j \in S_2} {A_j}^c)},
\end{align*}
where we have used \eqref{first cond} and \eqref{LLL1} on the last line. 

Hence it suffices to show that $\prod_{j \in S_1}(1 - x_j) \leq P(\cap_{j \in S_1} {A_j}^c | \cap_{j \in S_2} {A_j}^c)$.
Assuming that $[0] = \emptyset$, a telescopic product argument gives
\begin{align*}
P(\cap_{j \in S_1} {A_j}^c | \cap_{j \in S_2} {A_j}^c) &= \frac{P((\cap_{k \in [r]} {A_{j_k}}^c) \cap (\cap_{j \in S_2} {A_j}^c))}{P(\cap_{j \in S_2} {A_j}^c)} \\
&= \prod_{t=1}^{r} \frac{P((\cap_{k \in [t]} {A_{j_k}}^c) \cap (\cap_{j \in S_2} {A_j}^c))}{P((\cap_{k \in [t-1]} {A_{j_k}}^c) \cap (\cap_{j \in S_2} {A_j}^c))}\\
&= \prod_{t=1}^{r} P({A_{j_t}}^c | \cap_{j \in S_2 \cup \{j_1, \ldots, j_{t-1}\}}{A_j}^c) \\
&= \prod_{t=1}^{r}\left( 1 - P({A_{j_t}} | \cap_{j \in S_2 \cup \{j_1, \ldots, j_{t-1}\}}{A_j}^c) \right) \\
&\geq \prod_{t=1}^{r}(1 - x_{j_t}) = \prod_{j \in S_1}(1 - x_j),
\end{align*}
completing the proof. \qedhere
\end{claimproof}
\end{proof}

\section{Construction of a negative dependency graph for random injections}
In this section, we will construct a negative dependency graph for a collection of events in the space of random injections from one set to another. This construction will then be used to prove results on hypergraph packing and Latin transversals using the Lov\'asz Local Lemma. The construction is from \cite{Lu/Szekely}.

Let $U$ and $V$ be two finite sets with $|U| \leq |V|$. Let $\Omega = I(U,V)$ denote the set of all injective maps from $U$ to $V$, equipped with a uniform probability measure. We define a \text{matching} to be a triple $(S, T, f)$ satisfying:
\begin{enumerate}
\item[(i)] $S, T \text{ are subsets of } U, V \text{ respectively, and}$
\item[(ii)] $f: S \rightarrow T \text{ is a bijection.}$
\end{enumerate}

We denote the set of all such matchings by $M(U,V)$. By a canonical identification, we have $\Omega = I(U,V) \subseteq M(U,V)$. For any permutation $\rho: V \rightarrow V$ of $V$, define the map $\pi_\rho : M(U,V) \rightarrow M(U,V)$ by $(\pi_\rho(g))(u) = \rho(g(u)) \text{ for all } u \in \text{Domain}(g)$. Hence $\pi_\rho(g) = \rho \circ g$. This map induces a natural map on the set of matchings by defining $\pi_\rho((S, T, g)) = (S, \rho(T), g')$, where $g'(u) = \rho(g(u)) \text{ for all } u \in S$. It is easy to see that $\pi_\rho: M(U,V) \rightarrow M(U,V)$ is an injection for each permutation $\rho$ of $V$.

Two matchings $(S_1, T_1, f_1)$ and $(S_2, T_2, f_2)$ are said to \textbf{conflict} each other if either $f_1(k) \neq f_2(k) \text{ for some } k \in S_1 \cap S_2$ or if $f_1^{-1}(k) \neq f_2^{-1}(k)
$ for some $k \in T_1 \cap T_2$. Thus, two non-conflicting matchings can be naturally combined to get a ``bigger'' matching. 

For a given matching $(S, T, f)$, we define the event 
\begin{equation*}
A_{S,T,f} = \{\sigma \in \Omega: \sigma(i) = f(i) ~\forall ~ i \in S\}.
\end{equation*} 
Thus, $A_{S,T,f}$ contains injections from $U$ to $V$ that extend the map $f$. An event is said to be \textbf{canonical} if $A = A_{S,T,f}$ for some matching $(S,T,f)$. It is not very difficult to see that such a canonical representation for an event is unique, except when $|U| = |V|$ and $|A| = 1$. Due to this property of having unique non-trivial representations, the notion of conflicting matchings can be extended to canonical events in a well-defined way. We say that two canonical events conflict each other if their associated matchings conflict. We now make some immediate observations.

\begin{remark}
Two canonical events conflict if and only if they are disjoint.
\end{remark}
\begin{remark}
If $\rho : V \rightarrow V$ is any permutation, and $(S,T,f)$ is a matching, then $\pi_\rho(A_{S,T,f}) = A_{S,\rho(T), f'}$, where $f': S \rightarrow \rho(T)$ is defined by $f'(u) = \rho(f(u))$, as before.
\end{remark}
We now state the result constructing a negative dependency graph for a collection of canonical events. The reader can refer to \cite{Lu/Szekely} for the technical proof.

\begin{theorem}
Let $A_1, \ldots, A_n$ be a collection of canonical events in $\Omega = I(U,V)$. Let $G$ be the graph on $[n]$ with the edge-set
\begin{equation*}
E(G) = \{\{i,j\}: A_i \text{ and } A_j \text{ conflict}\}.   
\end{equation*}
Then $G$ is a negative dependency graph for the events $A_1, \ldots A_n$.
\end{theorem}

\section{Applications to packing of hypergraphs}
The negative dependency graph constructed in the previous section can be used in many situations where the events that need to be avoided can be identified with canonical events in a space of injections. In this section, we will demonstrate this with some results on hypergraphs, and we will continue this theme in the next section that covers a result on existence of Latin transversals. The results in this section are from \cite{Lu/Szekely}.

Recall that a hypergraph $H$ is a pair $(V(H), E(H))$ where $V(H)$ is a collection of points that we call the \textit{vertices} of $H$, and $E(H)$ is a collection of subsets of $V(H)$ called \textit{edges} of $H$. A hypergraph is called $\mathbf{r}$\textbf{-uniform} if all of its edges have cardinality equal to $r$. The \textbf{degree} of a vertex is defined as the number of edges that contain that vertex. Thus the \textbf{complete} $r$\textbf{-uniform hypergraph}, denoted $K_n^{(r)}$, is the hypergraph on a set of $n$ vertices whose edges are precisely all $r$-subsets of the vertex-set. A hypergraph $G$ is called a \textbf{subhypergraph} of a hypergraph $H$ if $V(G)$ is a subset of $V(H)$, and $E(G)$ is a subset of the collection of all edges of $H$ that contain elements from $V(G)$. 

We say that a collection of $r$-uniform hypergraphs $H_1, \ldots, H_k$ can be \textbf{packed into} an $r$-uniform hypergraph $H$, if there exist injections of $V(H_1) \ldots, V(H_k)$ into $V(H)$ such that the natural images of edge-sets are disjoint. Two hypergraphs $H_1$ and $H_2$ are called \textbf{isomorphic} if there is a bijection from $V(H_1)$ to $V(H_2)$ such that the natural image of the edge-set of $H_1$ coincides with the edge-set of $H_2$. For two $r$-uniform hypergraphs $H$ and $G$, we say that $H$ has a \textbf{perfect} $\mathbf{G}$\textbf{-packing}, if there exist vertex-disjoint hypergraphs $G_1, \ldots, G_k$ of $H$, each isomorphic to $G$, such that the sets $V(G_1), \ldots, V(G_k)$ partition $V(H)$. An obvious necessary condition for $H$ to have a perfect $G$-packing is for $|V(G)|$ to divide $|V(H)|$. We will derive a sufficient condition after we prove the following theorem on packing into complete $r$-uniform hypergraphs.
\vspace{-1mm}
\begin{theorem}
For each $i \in \{1,2\}$ let $H_i$ be an $r$-uniform hypergraph with $m_i$ edges such that each edge in $H_i$ intersects at most $d_i$ other edges of $H_i$. If $(d_1 + 1)m_2 + (d_2 + 1)m_1 < \frac{1}{e}\binom{n}{r}$, then $H_1$ and $H_2$ can be packed into $K_n^{(r)}$.
\end{theorem}
\begin{proof}
Without loss of generality, we assume that $H_2$ is given as a subhypergraph  of $K_n^{(r)}$. Let $U$ be the vertex-set of $H_1$, and let $V$ be that of $K_n^{(r)}$. We consider the probability space $\Omega = I(U,V)$ with uniform probability measure. We are following the notation of Section 3. Hence, $\Omega$ is just the set of all injections from $U$ to $V$, each injection having the same probability of occurring. It is clear that $H_1$ can be embedded into $K_n^{(r)}$ in a way that none of the edges of $H_2$ coincide with image of an edge of $H_1$ if none of the events $A_{F_1,F_2,\phi}$ (where $F_1, F_2$ are edges of $H_1, H_2$ respectively, and $\phi: F_1 \rightarrow F_2$ is a bijection) happen. Define $\mathcal{A} = \{A_{F_1,F_2, \phi} : F_1 \in E(H_1), F_2 \in E(H_2), \phi : F_1 \rightarrow F_2 \text{ is a bijection}\}$. We have the negative dependency graph $G$ for these events as described in the previous section.

Let $|U| = m$. Then, the number of injections from $U $ to $V$ is clearly equal to $\binom{n}{m} \cdot m!$. Now fix edges $F_1$, $F_2$ of $H_1, H_2$ respectively and a bijection $\phi: F_1 \rightarrow F_2$. The number of elements in the event $A_{F_1,F_2,\phi}$ is $\binom{n-r}{m-r} \cdot (m-r)!$. Indeed, the action of any map in this event is already determined on the $r$ elements of $F_1$, and the remaining $(m-r)$ elements of $U$ can be mapped injectively into the remaining $(n-r)$ elements of $V$ in so many ways. Hence,
\begin{align*}
P(A_{F_1, F_2, \phi}) = \frac{|A_{F_1, F_2, \phi}|}{|\Omega|} = \frac{\binom{n-r}{m-r} \cdot (m-r)!}{\binom{n}{m} \cdot m!} = \frac{1}{r! \binom{n}{r}} = p \text{ (say)}.
\end{align*}
Now $A_{F_1,F_2,\phi}$ conflicts with another canonical event $A_{F_1',F_2',\phi'}$ if and only if one of the following happens:
\begin{enumerate}
\item[(i)] $F_1 \cap F_1' = \emptyset$ and $F_2 \cap F_2' \neq \emptyset$.
\item[(ii)]$F_1 \cap F_1' \neq \emptyset$ and $\phi(x) = \phi'(x)$ for some $x \in F_1 \cap F_1'$.
\end{enumerate}
Let $c_1$ denote the number of canonical events $A_{F_1',F_2',\phi'}$ that conflict with $A_{F_1,F_2,\phi}$ due to (i), and let $c_2$ denote the corresponding number due to (ii). Simple counting arguments show that
\begin{align*}
c_1 \leq m_1 (d_2 + 1) r! - 1 \text{ and } c_2 \leq (d_1 + 1) m_2 r! - 1. 
\end{align*}
Hence, the degree $d$ of the negative dependency graph $G$ satisfies the inequality $d \leq c_1 + c_2 \leq r![(d_1 + 1)m_2 + (d_2 + 1)m_1] - 1$. Then,
\begin{align*}
ep(d+1) \leq e \frac{1}{r! \binom{n}{r}} r![(d_1 + 1)m_2 + (d_2 + 1)m_1] < 1 \text{ by hypothesis.}
\end{align*}
Hence, by Lemma 2.2, $P(\cap_{A \in \mathcal{A}}A) >0$, which implies that $\cap_{A \in \mathcal{A}}A \neq \emptyset$, as desired.
\end{proof}

Using this result, we can now prove the following sufficient condition for existence of a perfect $G$-packing. 

\begin{theorem}
Suppose $G$ and $H$ are two $r$-uniform hypergraphs that satisfy the following:
\begin{enumerate}
\item[(i)] $G$ has $s$ vertices, $H$ has $n$ vertices, and $s$ divides $n$.
\item[(ii)] $G$ has $m$ edges and each edge in $G$ intersects at most $d$ other edges of $G$. 
\item[(iii)] Each vertex of $H$ has degree greater than or equal to $(1-x)\binom{n-1}{r-1}$ for some number $x$.
\end{enumerate}
If $x < \frac{1}{e(d + 1 + r^2\frac{m}{s})}$, then $H$ has a perfect $G$-packing.
\end{theorem}
\begin{proof}
Let $H_1$ be the union of $\frac{n}{s}$ vertex-disjoint copies of $G$. Let $H_2$ be the complement $r$-uniform hypergraph of $H$; that is, $V(H_2) = V(H_1)$ and $E(H_2)$ consists of all $r$-subsets of $V(H)$ that are not edges of $H$. It is not very difficult to see that $H$ has a perfect $G$-packing if and only if $H_1$ and $H_2$ can be packed into $K_n^{(r)}$. Following the notation of the previous theorem, we have, 
\begin{align}\label{first data}
d_1 = d \text{ and } m_1 = |E(H_1)| = \frac{n}{s}|E(G)| = \frac{nm}{s}.
\end{align}
Now, condition (iii) implies that the degree of any vertex in $H$ is at least $(1-x)\binom{n-1}{r-1}$. Also, given an element of $V(H)$, there are $\binom{n-1}{r-1}$ $r$-subsets of $V(H)$ containing that vertex. Hence the degree of any vertex in $H_2$ is at most $\binom{n-1}{r-1} - (1-x)\binom{n-1}{r-1} = x \cdot \binom{n-1}{r-1}$. Hence, for an edge $F$ of $H_2$, the number of other edges that intersect $H$ is at most equal to $\sum_{v \in F}(d_{H_2}(v) - 1) = r[x \binom{n-1}{r-1} - 1] \leq rx \binom{n-1}{r-1} - 1$, where $d_{H_2}(v)$ denotes the degree of $v$ in $H_2$. Since this is valid for any edge of $H_2$, we have
\begin{equation}\label{second data}
d_2 \leq rx \binom{n-1}{r-1} - 1.
\end{equation}
Also, we have 
\vspace{-2mm}
\begin{align*}
|E(H)| &\geq \sum_{v \in V(H)} d_{H}(v) \geq n(1-x)\binom{n-1}{r-1} = r(1-x)\binom{n}{r}\\
&\geq (1-x)\binom{n}{r}.
\end{align*}
Hence, we get 
\vspace{-2mm}
\begin{align}\label{third data}
m_2 = |E(H_2)| = \binom{n}{r} - |E(H)| \leq \binom{n}{r} - (1-x)\binom{n}{r} = x \binom{n}{r}.
\end{align}
From, \eqref{first data}, \eqref{second data}, and \eqref{third data}, we have
\begin{align*}
(d_1 + 1)m_2 + (d_2 + 1)m_1 &\leq (d+1)x\binom{n}{r} + rx\binom{n-1}{r-1}\cdot\frac{nm}{s} \\
&\leq \binom{n}{r} \left[ d+1 + \frac{r^2m}{s}\right]x \\
&\leq  \frac{1}{e} \binom{n}{r}, \text{ by hypothesis.}
\end{align*}
Hence Theorem 4.1 applies, completing the proof.
\end{proof}

If $m=1$ (so that $d=0$ and $s=r$), then $G$ can be viewed as a single edge of $H$, and so a perfect $G$-packing of $H$ amounts to a partition of vertices of $H$ by edges of $H$. If such a partition exists, we say that $H$ has a \textbf{perfect matching}. In the further special case when $r = 2$, we get perfect matchings of graphs. Hence we have the following results.
\begin{corollary}
Suppose that $r$ divides $n$. If the degree of each vertex in an $r$-uniform hypergraph $H$ on $n$ vertices is at least $(1 - \frac{1}{e(1+r)})\binom{n-1}{r-1}$, then $H$ has a perfect matching.
\end{corollary}

\begin{corollary}
If the degree of any vertex in a graph $G$ is at least $\frac{(3e-1)(n-1)}{3e}$, then $G$ has a perfect matching.
\end{corollary}

\section{A result on Latin transversals}
Let $A = ((a_{i,j}))$ be an $(n \times n)$ matrix. A permutation $\pi$ of $[n]$ is called a \textbf{Latin transversal} of $A$ if the entries $a_{i, \pi(i)}$ for $i \in [n]$ are all distinct. The Lopsided Lov\'asz Local Lemma was initially introduced in \cite{Erdos/Spencer} to obtain a sufficient condition for existence of Latin transversals. We will use the new terminology built in Section 3 to get a slightly improved version of their result.

\begin{theorem}
Suppose $k \leq \frac{n-1}{4e}$, and suppose that no element appears in more than $k$ entries of $A$. Then $A$ has a Latin transversal. 
\vspace{-1mm}
\end{theorem}
\begin{proof}
Let $\Omega = I([n],[n])$ with uniform probability measure. Hence, $\Omega$ is just the set of all permutations of $[n]$, so that $|\Omega| = n!$. Let $\mathcal{B} = \{(\{i,i'\},\{j,j'\}): i < i', j \neq j', a_{i,j} = a_{i',j'}\}$. If $I = \{i, i'\}, J = \{j,j'\}$ are such that $(I,J) \in \mathcal{B}$ with $i < i'$ and $a_{i,j} = a_{i',j'}$, then define $\pi_{(I,J)} : I \rightarrow J$ by $\pi_{(I,J)}(i) = j$, and $\pi_{(I,J)}(i') = j'$. It is not too difficult to see that there is a Latin transversal for $A$ if and only if there is a permutation that is not an element of $A_{I,J,\pi_{(I,J)}}$ for any $(I,J) \in \mathcal{B}$. Hence it suffices to show that $P(\cap_{(I,J) \in \mathcal{B}}{A_{I,J, \pi_{I,J}}}^c) > 0$. 

Let $\mathcal{A} = \{A_{I,J,\pi_{(I,J)}} : (I,J) \in \mathcal{B}\}$. Let $G$ be the negative dependency graph for these canonical events, as constructed in Section 4. Let $d$ be the maximum degree of a vertex of $G$, and suppose that it is the degree of the vertex corresponding to the event $A_{I,J,\pi_{(I,J)}}$. Note that for $(I_1,J_1) \in \mathcal{B}\backslash\{(I,J)\}$, $A_{I,J,\pi_{(I,J)}}$ can possibly conflict with $A_{I',J',\pi_{(I',J')}}$ only if $I \cap I' \neq \emptyset$ or $J \cap J' \neq \emptyset$. For such a tuple $(I_1, J_1)$ with $I_1 = \{i_1, i_1'\}$ and $J_1 = \{j_1, j_1'\}$, there are $2n + 2n = 4n$ possible values of $(i_1, j_1)$ (since $i_1 \in I$ or $j_1 \in J$), and for each of these values of $(i_1, j_1)$, there are $k$ possible values of $(i_1', j_1')$, since we require that $a_{i_1, j_1} = a_{i_1', j_1'}$, and an entry is repeated at most $k$ times in the matrix $A$. Hence
\begin{align*}
\vspace{-2.5mm}
d &\leq |\{(I_1,J_1): I \cap I_1 \neq \emptyset \text{ or } J \cap J_1 \neq \emptyset\}\backslash\{(I,J)\}| \\
&= |\{(I_1,J_1): I \cap I_1 \neq \emptyset \text{ or } J \cap J_1 \neq \emptyset\}| - 1 \\
&\leq 4nk - 1.
\vspace{-3mm}
\end{align*}
It is clear that for any $(I,J) \in \mathcal{B}$, $|A_{I,J,\pi_{(I,J)}}| = (n-2)!$, as the action of a permutation in this set is already specified on the set $I$ of two elements. Hence, for any $(I,J) \in \mathcal{B},$
\begin{equation*}
P(A_{I,J,\pi_{(I,J)}}) = \frac{(n-2)!}{n!} = \frac{1}{n(n-1)} = p  ~\text{(say)}.
\vspace{-3mm}
\end{equation*}
Then, $ep(d+1) \leq \frac{e(4nk)}{n(n-1)} = e \cdot \frac{4k}{n-1} \leq 1$.
Hence, the conditions of Lemma 2.2 are satisfied, so that $P(\cap_{A \in \mathcal{A}}A^c) >0$, completing the proof.
\end{proof}
\section*{Acknowledgements}
The author is grateful to Stefan van Zwam for his supervision and guidance in this project as part of the course \textit{Communicating Mathematics II} at Louisiana State University.

\end{document}